\documentclass[12pt]{amsart}

\usepackage{comment}
\newcommand{\disp}{\displaystyle}
\newcommand{\nc}{\newcommand}
\nc{\G}{{\Gamma}} \nc{\BC}{{\mathbb C}} \nc{\BQ}{{\mathbb Q}}
\nc{\BR}{{\mathbb R}} \nc{\BZ}{{\mathbb Z}} \nc{\BP}{{\mathbb P}} \nc{\PC}{{\BP_1(\BC)}}
\nc{\BN}{{\mathbb N}} \nc{\BM}{{\mathbb M}}
\nc{\fH}{{\mathbb H}}

\nc{\mat}{{\binom{a\,\ b}{c\,\ d}}}
\nc{\U}{{\mathcal U}}
\nc{\PS}{{\mbox{PSL}_2(\BZ)}} \nc{\SL}{{\mbox{SL}_2(\BZ)}}
\nc{\SR}{{\mbox{SL}_2(\BR)}} \nc{\PR}{{\mbox{PSL}_2(\BR)}}
\nc{\SLC}{{\mbox{SL}_2(\BC)}}

\nc{\GL}{{\mbox{GL}}} \nc{\PQ}{{\mbox{PGL}_2^+(\BQ)}}
\nc{\GR}{{\mbox{GL}_2^+(\BR)}} \nc{\PG}{{\mbox{PGL}_2(\BC)}}
\nc{\GC}{{\mbox{GL}_2(\BC)}}
\nc{\f}{{\mathcal{F}(\fH)}}
\nc{\Cc}{\widehat{\BC}}
\nc{\e}{{E_{\varrho}(\G)}}
\nc{\g}{{\gamma}}
\nc{\vm}{{V_{\varrho}(\G)}}
\nc{\oo}{{\mathcal O}}
\nc{\M}{{\mbox{M}}}
\nc{\om}{{\omega}}
\nc{\Om}{{\Omega}}
\nc{\TX}{{\widetilde{X}}}
\nc{\ol}{\overline}
\nc{\cl}{{\mathcal L}}
\nc{\ce}{{\mathcal E}}
\nc{\la}{{\lambda}}
\nc{\La}{{\Lambda}}

\nc{\cz}{{\mathcal Z}}

\newtheorem{numbered}{}[section]
\newtheorem{thm}[numbered]{Theorem}

\numberwithin{equation}{section}

\newcommand{\thmref}[1]{Theorem~\ref{#1}}

\begin{document}
	
	\title[]{Hypergeometric solutions to Schwarzian equations}
	\author[]{Khalil Besrour} \author[]{Abdellah Sebbar}

	\address{Department of Mathematics and Statistics, University of Ottawa,
		Ottawa Ontario K1N 6N5 Canada}
	\email{kbesr067@uottawa.ca}
	\email{asebbar@uottawa.ca}
\maketitle
\begin{abstract} In this paper we study the modular differential equation $y''+s\,E_4\, y=0$ where $E_4$ is the weight 4 Eisenstein series and $s=\pi^2r^2$ with $r=n/m$ being a rational number in reduced form such that $m\geq 7$. This study is carried out by solving the associated Schwarzian equation $\{h,\tau\}=2\,s\,E_4$ and using the theory of equivariant functions on the upper half-plane and the 2-dimensional vector-valued modular forms. The solutions are expressed in terms of the Gauss hypergeometric series. This completes the study of the above-mentioned modular differential equation of the associated Schwarzian equation given that the cases $1\leq m\leq 6$ have already been treated in \cite{forum,ramanujan, jmaa,preprint}.
\end{abstract}
%%%%%%%%%%%%%%%%%%%%%%%%%%%%%%
\section{Introduction}
%%%%%%%%%%%%%%%%%%%%%%%%%%%%%%%
A second order modular differential equation of weight $k\in\BZ$ is, according to \cite{ka-ko,ka-za}, a differential equation on $\fH=\{\tau\in\BC\,:\ \mbox{Im}(\tau)>0\}$ of the form
\[
y''+A(\tau)\,y'+B(\tau)\,y\,=\,0\ ,\ \ \tau\in\fH\,,
\]
where $A(\tau)$ and $B(\tau)$ are holomorphic on $\fH$ with  specific boundedness conditions when $\mbox{Im}\,(\tau)\rightarrow \infty$ and such that the space of solutions is invariant under the transformation $y(\tau)\mapsto (c\tau+d)^{-k}y(\gamma\tau)$, where $\gamma=\binom{a\ \ b}{c\ \  d}\in\SL$. %Here, the differentiation $'$ stands for $ \frac{1}{2\pi i}\frac{d}{d\tau}$. 
This equation can be reduced to its normal form $y''+C(\tau)y=0$ where $C(\tau)$ is necessarily a holomorphic weight 4 modular form and thus takes the shape
\begin{equation}\label{eq1}
y''+s\,E_4\,y\,=\,0,
\end{equation}
where $E_4$ is the weight 4 Eisenstein series and $s$ is a complex parameter. This differential equation becomes modular of weight -1.
In this paper we focus on the case  $s=\pi^2r^2$ where $r= n/m$ is a rational number with $\gcd(m,n)=1$ and $m\geq 7$. In fact, this equation has been studied  for the case $m=1$ in \cite{jmaa}, for the cases $2\leq m\leq 5$ in \cite{forum}. The case $m=6$ and $n\equiv 1 \mod 12$ was solved in \cite{ramanujan} and then completed to all $n$ in \cite{preprint}. The nature of solutions differs from case to case and involves a different set of tools and techniques as it will be seen below. The equation \eqref{eq1} is closely related to the Schwarz differential equation
\begin{equation}\label{eq2}
\{h,\tau\}\,=\,2s\,E_4(\tau)
\end{equation}
where the unknown function $h$ is a meromorphic function on $\fH$ and
$\{h,\tau\}$ is the Schwarz derivative defined by
\[
\{h,\tau\}\,=\, \left(\frac{h''(\tau)}{h'(\tau)}\right)'\,-\,\frac{1}{2}\left(\frac{h''(\tau)}{h'(\tau)}\right)^2.
\]
The relationship between \eqref{eq1} and \eqref{eq2} is as follows: If $y_1$ and $y_2$ are two linearly independent solutions to \eqref{eq1}, then $h=y_1/y_2$ is  a solution to \eqref{eq1}. Conversely, if $h$ is a solution to \eqref{eq2} which is locally univalent where it is holomorphic and has only simple poles (if any), 
then $y_1=h/\sqrt{h'}$ and $y_2=1/\sqrt{h'}$ are two linearly  independent solutions to \eqref{eq2}. In the meantime, the condition on $h$ taking its values only once in $\BC\cup\{\infty\}$ is equivalent to $\{h,\tau\}$ being holomorphic in $\fH$ \cite{mathann}. Therefore, since $E_4$ is holomorphic in $\fH$, we have a well-defined  correspondence
between the solutions of \eqref{eq1} and those of \eqref{eq2}.

Furthermore, using the properties of the Schwarz derivative, one can show that the Schwarz derivative of a meromorphic function $h$ on $\fH$ is a weight 4 automorphic form for a Fuchsian group $\G$ if and only if there exists a 2-dimensional complex representation $\varrho$ of $\G$ such that
\[
h(\gamma\cdot\tau)\,=\,\varrho(\gamma)\cdot h(\tau)\,,\ \ \tau\in\fH\,,\ \ \gamma\in\G,
\]
where the matrix action on both sides is by linear fractional transformations. The function $h$ is then called a $\varrho-$equivariant function for $\G$. As an example, if $F=(f_1,f_2)^t$ is a 2-dimensional vector-valued automorphic form with a multiplier system $\varrho$ for $\G$, then $h=f_1/f_2$ is $\varrho-$equivariant. Also, if $f$ is a scalar automorphic form of weight $k$ for $\G$, then
\[
h_f(\tau)\,=\,\tau\,+\,k\frac{f(\tau)}{f'(\tau)}
\]
is $\varrho-$equivariant for $\varrho=\mbox{Id}$, the defining representation of $\G$ \cite{rational}. We simply refer to it as an equivariant function for $\G$.

We now focus on the case $\G=\SL$ and we suppose that for a meromorphic function $h$ on $\fH$, $\{h,\tau\}$ is a holomorphic weight 4 modular form for $\G$, that is, $\{h,\tau\}$ is a scalar multiple of $E_4$. It turns out that if we are looking for $h$ to be either  meromorphic at the cusps or having logarithmic singularities therein, then $\{h,\tau\}=2\pi^2r^2E_4$ with $r\in\BQ$. The essential facts of \cite{forum,ramanujan,jmaa} can be summarized as follows:

The case $r\in\BZ$ corresponds to solutions that are equivariant functions ($\varrho=\mbox{Id}$) given by quasi-modular forms. The case where $\varrho$ is irreducible with finite image corresponds to $r=n/m$ with $2\leq m\leq 5$ and the solution $h$ to \eqref{eq2} is a modular function for $\mbox{Ker}\,\varrho=\G(m)$, the principal congruence group of level $m$. The integers $m$ and $n$ respectively represent the degrees of  the following two coverings of compact Riemann surfaces 
\[
h:X(\ker\varrho)\longrightarrow X(\SL)\cong \PC
\] induced by the solution $h$ and
 \[
\pi:X(\ker\varrho)\longrightarrow X(\SL)\cong \PC
\]
 induced by the natural inclusion $\ker\varrho\subseteq\SL$.

In the meantime, if $\varrho$ is reducible then necessarily $m=6$ whence the solution to \eqref{eq2} is given by the integral of a weight 2 differential form on the Riemann surface $X(\SL)$. The level 6 is distinguished mainly due  to the fact that the commutator group of $\PS$ is an index 6 subgroup. Notice that in all these cases when $m>1$, $\varrho(T)$ has a finite order equal to $m$ where $ T=\binom{1\ \ 1}{0\ \ 1}$. 

We are thus left with the case of irreducible representations $\varrho$ of $\G$ with infinite image, that is, when $m\geq 7$. The purpose of this paper is to construct solutions to \eqref{eq1} and \eqref{eq2} in these cases by means of hypergeometric series using the works of Franck and Mason \cite{fr-ma} and of Mason \cite{mason} on vector-valued modular forms.

%%%%%%%%%%%%%%%%%%%%%%%%%%%%%%%%%%%%%%%%%%

%%%%%%%%%%%%%%%%%%%%%%%%%%%%%%%%%%%%%%%%%%
\section{Two-dimensional vector-valued modular forms}
%%%%%%%%%%%%%%%%%%%%%%%%%%%%%%%%%%%%%%%%%%%%%%%%%%
Recall the Eisenstein series 
\begin{align*}
E_2(\tau)\,&=\,1-24\sum_{n\geq 1}\,\sigma_1(n)q^n\,,\\
E_4(\tau)\,&=\,1+240\sum_{n\geq 1}\,\sigma_3(n)q^n\,,\\
E_6(\tau)\,&=\,1-504\sum_{n\geq 1}\,\sigma_5(n)q^n\,,
\end{align*}
where $q=\exp(2\pi i \tau)$, $\tau\in\fH$, and $\sigma_k(n)$ is the sum of the $k-$th powers of the positive divisors of $n$. Then $E_4$ and $E_6$ are modular forms of weights 4 and 6 respectively, while $E_2$ is a quasi-modular  of weight 2. We also recall the classical modular forms and functions:
\[
\eta(\tau)\,=\,q^{\frac{1}{24}}\,\prod_{n\geq1}(1-q^n)\,,
\]
the weight 12 cusp form 
\[
\Delta(\tau)\,=\,\eta(\tau)^{24}\,=\,\frac{1}{1728}(E_4(\tau)^3-E_6(\tau)^2),
\]
and the elliptic modular function 
\[
j(\tau)\,=\,\frac{1}{1728}\frac{E_4(\tau)^3}{\Delta}.
\]

Let $\varrho$ be a two-dimensional irreducible complex representation of the modular group for which $\varrho(T)$ is of finite order $m$. Irreducibility implies $m > 1$. Now $\varrho(T)$, being of finite order, is diagonalizable and hence, up to conjugacy,  it has the form 
$$ \varphi(T) = \begin{pmatrix}
\sigma & 0 \\
0 & \sigma'
\end{pmatrix} $$
where $\sigma$ and $\sigma'$ are $m$-th roots of unity. More generally, we have the following result 
\begin{thm}[\cite{mason2}, Theorem 3.1]\label{thm1}
    Let $\mu_1, \mu_2 \in \mathbb{C}$, $\mu_1\neq\mu_2$, such that $(\mu_1 \mu_2)^6=1$ and $\mu_1/\mu_2$ is not a primitive $6$-th root of unity. Then there exists a unique irreducible $2$-dimensional representation $\varrho$ of $\Gamma$ such that
    $$ \varrho(T) = \begin{pmatrix}
\mu_1 & 0 \\
0 & \mu_2
\end{pmatrix}. $$
\end{thm}

The space of vector-valued modular forms with respect to a representation $\varrho$ of the modular group $\Gamma$ is denoted by $H(\varrho)$. It is a graded module with respect to the weights of the modular forms. We denote by $H_k(\varrho)$ the subspace of elements of $H(\varrho)$ of weight $k$. The operator $D_k := \frac{d}{d\tau} - \frac{k}{12} E_2$ maps $H_k(\varrho)$ into $H_{k+2}(\varrho)$. Also,
$H(\varrho)$ has the structure of a free module over the ring of scalar modular forms $\mathbb{C}[E_4,E_6]$ of rank dim($\varrho$) \cite{mason2}. In the 2-dimensional case, we have the following result.

\begin{thm}[\cite{mason2}, Theorem 5.5]\label{thm2}
    Let $\varrho$ be a $2-$dimensional irreducible representation of $\Gamma$ such that
    $$ \varrho(T) = \begin{pmatrix}
e^{2\pi i a} & 0 \\
0 &e^{2\pi i b}
\end{pmatrix}
$$
for $0 \leq b < a < 1$ real numbers. There exists $F_0 \in H(\varrho)$ of weight $k = 6(a+b)-1$ such that
$$ H(\varrho) = \mathbb{C}[E_4,E_6]F_0 \oplus \mathbb{C}[E_4,E_6]D_kF_0.$$
Moreover, $F_0$ has the $q-$expansion
  \begin{align*}
    F_0(\tau) &= \begin{bmatrix}
           f_1(\tau) \\
           f_2(\tau)
         \end{bmatrix} = \begin{bmatrix}
           q^a \sum\limits_{n = 0}^\infty a_n q^n \\
           q^b \sum\limits_{n = 0}^\infty b_n q^n \\
         \end{bmatrix} 
  \end{align*}
with $a_0=b_0=1$.
\end{thm}

If $\varrho$ is a fixed   irreducible representation, then $F_0$ is called the vector-valued modular form of minimal weight. In \cite[Section 4.1]{fr-ma}, the components of $F_0$ are computed in terms of hypergeometric series
$$ f_1 = \eta^{2k} \left(\frac{1728}{j}\right)^{\frac{a-b}{2}+\frac{1}{12}}F\left(\frac{a-b}{2}+\frac{1}{12}, \frac{a-b}{2}+\frac{5}{12} ; a-b + 1 ; \frac{1728}{j}\right)$$
and
$$ f_2 = \eta^{2k} \left(\frac{1728}{j}\right)^{\frac{b-a}{2}+\frac{1}{12}}F\left(\frac{b-a}{2}+\frac{1}{12}, \frac{b-a}{2}+\frac{5}{12} ; b-a + 1 ; \frac{1728}{j}\right).$$
Here $F$ is the Gauss hypergeometric series defined by
\[
 F(a,b;c;z) := 1 + \sum_{n \geq 1} \frac{(a)_n \, (b)_n}{(c)_n} \frac{z^n}{n!}\ ,\quad (a)_n := a (a+1) \cdots (a+n-1).
\]
  \smallskip

\section{Wronskian of a vector-valued modular form}
Let $y_1$ and $y_2$ be two linearly independent solutions to \eqref{eq1} on $\fH$. Their existence is guaranteed since $E_4$ is holomorphic and $\fH$ is simply connected. If $h=y_1/y_2$ is the corresponding solution to \eqref{eq2}, then one can see that $\{h,\tau\}$ is holomorphic in $\fH$ if and only if the Wronskian $W(y_1,y_2)=y'_1y_2-y_1y'_2$ is nowhere vanishing on $\fH$. Indeed, we have $h'=W(y_1,y_2)/y_2^2$, and the holomorphy of $\{h,\tau\}$ is equivalent to $h'(\tau)$ being nonzero where $h$ is regular, and having only simple poles where it is meromorphic. Similarly, if we are given a vector-valued modular form $F=(f_1,f_2)^T$ of weight $k$ and multiplier system $\varrho$, then the $\varrho-$equivariant function $h=f_1/f_2$ has a holomorphic Schwarz derivative if and only if the Wronskian $W(F):=f'_1f_2-f_1f'_2$ is nowhere vanishing on $\fH$. In the meantime, we have the following important property of the Wronskian for a vector-valued modular form.

\begin{thm}[\cite{mason}, Theorem 3.7]\label{thm3}
    Let $F=(f_1,f_2)^T$ be a vector-valued modular form of weight $k$ with $q-$expansion $f_i = q^{a_i} + \mbox{O}(q^{a_i+1})$ \; for $i=1,2$. Then
    $$ W(F) = \Delta^{a_1+a_2} G$$
    where $G$ is a scalar modular form of weight $2(k+1)-12(a_1+a_2)$ that is not a cusp form.
\end{thm}

%%%%%%%%%%%%%%%%%%%%%%%%%%%%%%%%%%%%%%%%%%%%%%%
\section{The solutions}
%%%%%%%%%%%%%%%%%%%%%%%%%%%%%%%%%%%%%%%%%%%%%%%%

We know suppose that $m\geq 7$. If an irreducible 2-dimensional representation $\varrho$ of $\G$ is such that $\varrho(T)$ has order $m$, then necessarily Im$\,\varrho$ has an infinite image \cite{forum}.
Let $h$ be a solution to
 $$ 
 \{h,\tau\} = 2\pi^2\left(\frac{n}{m}\right)^2 E_4(\tau),
 $$
 where $m,n$ are positive integers with $m\geq 7$.
 %$m \geq 2$ and if $m=6$ then $n \ne 1,5$ modulo $6$.
The existence of $h$ is guaranteed by the existence of global solutions of the corresponding ODE \eqref{eq1}. The holomorphy of  $\{h,\tau\}$ allows us to define two functions $y_1=h/\sqrt{h'}$ and $y_2=1/\sqrt{h'}$ that are holomorphic solutions to \eqref{eq1}, see \cite{jmaa}. Moreover,
 the vector valued function 
$$ F_h = \begin{bmatrix}
           \frac{h}{\sqrt{h'}} \\
           \frac{1}{\sqrt{h'}}
         \end{bmatrix} = \begin{bmatrix}
           q^\frac{n}{2m}\sum\limits_{i=0}^\infty a_i q^i \\
           q^\frac{-n}{2m}\sum\limits_{i=0}^\infty b_i q^i
         \end{bmatrix} $$
is a weakly holomorphic vector-valued modular form of weight $-1$ with respect to a representation $\varrho$ that verifies 
\begin{equation}\label{cond} 
\varrho(T) = \begin{pmatrix}
e^{2\pi i \frac{n}{2m}} & 0 \\
0 & e^{-2\pi i \frac{n}{2m}}
\end{pmatrix}. 
\end{equation}

We now provide the solutions to the differential equation by constructing vector-valued weakly holomorphic modular forms of weight $-1$ with respect to a suitable irreducible representation that satisfies the conditions of \thmref{thm1}.

\begin{thm}\label{thm4.1}
    Suppose $n < m $ and $\gcd(m,n)=1$. Let $F_0=(f_1,f_2)^T$ be the 2-dimensional vector-valued  modular form of minimal weight  with respect to the unique irreducible representation $\varrho$ such that
$$ \varrho(T) = \begin{pmatrix}
e^{2\pi i \frac{m+n}{2m}} & 0 \\
0 & e^{2\pi i \frac{m-n}{2m}}
\end{pmatrix}. $$
 Then $h=f_1/f_2$ verifies
        $$ \{h,\tau\} \,=\, 2\pi^2
        \left(\frac{n}{m}\right)^2 E_4(z).$$
\end{thm}

\begin{proof} Since $m\geq 7$ and $1\leq n<m$, it is clear that the diagonal terms of $\varrho(T)$ satisfy the conditions of \thmref{thm1}, and thus provide the existence of a unique irreducible representation $\varrho$ such that $\varrho(T)$ is as stated. 
Let $F_0=(f_1,f_2)^T$ be the vector-valued modular form of minimal weight, which is then equal to 5, attached to $\varrho$, then $h=f_1/f_2$ is $\varrho-$equivariant. Therefore, the Schwarz derivative $\{h,\tau\}$ is a weight 4 (meromorphic) modular form, which we will now show that it is holomorphic on $\fH$ and at the cusps.
 By \thmref{thm2},  $F_0$ has the $q-$expansion 
 $$ F_0 =  \begin{bmatrix}
           q^\frac{m+n}{2m}\,\sum\limits_{i=0}^\infty \,a_i q^i \\
           q^\frac{m-n}{2m}\,\sum\limits_{i=0}^\infty \,b_i q^i
         \end{bmatrix}$$
 where the $a_i, b_i \in \mathbb{C}$ and $a_0 = b_0 = 1$. Hence, one can easily compute that  $\{h,\tau\}=2\pi^2(\frac{n}{m})^2 + \mbox{O}(q)$ which is holomorphic at $\infty$. In addition, according to \thmref{thm3}, the Wronskian of $F_0$ can be written as $W(F_0)=\Delta G$, where $G$ is a modular form of weight 0 since $F_0$ has weight 5, and thus $G$ is a nonzero constant $c$, that is, $W(F_0)=c\Delta$. It follows that $W(F_0)$ is nowhere vanishing in $\fH$, and as a consequence, $\{h,\tau\}$ is holomorphic on $\fH$. As the space of weight 4 modular forms is one-dimensional generated by $E_4$, and comparing the leading terms, one gets $\{h,\tau\}=2\pi^2(n/m)^2\,E_4(\tau)$.

\end{proof}

\

Having described $f_1$ and $f_2$ in terms of hypergeometric series in Section $2$, we finally have
\begin{thm} Let $m$ and $n$ be integer such that $m\geq 7$, $0<n<m$ and $\gcd(m,n)=1$. Then a solution to $\{h,\tau\}=2\pi(n/m)^2E_4(\tau)$ is given by
$$ h =  \left(\frac{1728}{j}\right)^{\frac{n}{m}} \frac{F\left(\frac{n}{2m}+\frac{1}{12}, \frac{n}{2m}+\frac{5}{12} ; \frac{n}{m} + 1 ; \frac{1728}{j}\right)}{F\left(\frac{-n}{2m}+\frac{1}{12}, \frac{-n}{2m}+\frac{5}{12} ; \frac{-n}{m} + 1 ; \frac{1728}{j}\right)}.$$
%Moreover, this solution is $\varrho-$equivariant for $\SL$ with $\varrho$ being irreducible with infinite image such that $\varrho(T)$ has order $m$.
Any other solution is a linear fraction of $h$.
\end{thm}\qed

We now proceed to construct the solutions for $n>m$ as well. The idea is to use both generators $F_0$ and $DF_0$ of the ring of vector-valued modular forms over $\BC[E_4,E_6]$. This will allow us to create modular forms of higher weight that give rise to solutions to our equation. Let $n$ be a positive integer such that $\gcd(m,n)=1$ and let $n'$ be the smallest positive residue of $n\mod m$. Let $F_0$ be the vector-valued modular form of minimal weight corresponding to the pair $(m,n')$ as in \thmref{thm4.1}. In this case, $F_0$ has weight 5 with the $q-$expansion
$$ F_0(\tau) =  \begin{bmatrix}
           q^\frac{m+n'}{2m}\sum\limits_{i=0}^\infty a_i q^i \\
           q^\frac{m-n'}{2m}\sum\limits_{i=0}^\infty b_i q^i
         \end{bmatrix} ,\ a_0=b_0=1\,,
         $$
and therefore 
$$ D_5F_0(\tau) = \begin{bmatrix}
           (\frac{m+n'}{2m} - \frac{5}{12})q^\frac{m+n'}{2m}(1 + 
           \mbox{O}(q))\\
           (\frac{m-n'}{2m} - \frac{5}{12})q^\frac{m-n'}{2m}(1 + \mbox{O}(q))
         \end{bmatrix}\,. $$
Now define $$F_1 := E_6F_0 - \frac{1}{\frac{m+n'}{2m} - \frac{5}{12}} E_4D_5F_0 =  
           \begin{bmatrix}
           c_1q^\frac{3m+n'}{2m}(1 + \mbox{O}(q))\\
           c_2q^\frac{m-n'}{2m}(1 + \mbox{O}(q))
         \end{bmatrix}$$
where 
\[
c_1=\frac{377m^2+2004mn'-2466{n'}^2}{(m-n')(m+6n')}
\]
and 
\[
c_2=\frac{12n'}{m+6n'}
\]
which are both non-zero for integers $m$ and $n'$ with $0<n'<m$. It is clear that $F_1=(g_1,g_2)^T$ is a modular form of weight $11$. Now applying \thmref{thm3}, we get that $W(F_1) = c \Delta^2$ and so  $\{g_1/g_2,\tau\} $ is holomorphic on $\fH$ and also at $\infty$ with $q-$expansion $\disp 2\pi^2(1+n'/m)^2(1 + \mbox{O}(q))$. It follows that $h=g_1/g_2$ solves 
$$ \{h,z\} =2\pi^2\left(\frac{n'}{m} + 1\right)^2 E_4(z).$$

The key to solving $\{h,z\} = 2\pi^2\left(\frac{n'}{m} + r\right)^2 E_4(z)$ is to iterate the above process $r$ times where $r$ is such that $n=rm+n'$.


\begin{thebibliography}{aaaa}
 %\bibitem{yang} Z. Chen; C-S. Lin; Y. Yang. Modular ordinary differential equations on $\SL$ of third order and applications. SIGMA Symmetry Integrability Geom. Methods Appl. 18 (2022), Paper No. 013.
		%\bibitem{structure} A. Elbasraoui; A. Sebbar. Equivariant forms: Structure and geometry. Canad. Math. Bull. Vol. {\bf 56} (3), (2013) 520--533.
		\bibitem{rational} A. Elbasraoui; A. Sebbar. Rational equivariant forms. Int. J. Number Th. 08  No. 4(2012), 963--981.
	\bibitem{fr-ma} C. Franc; G. Mason. Hypergeometric series, modular linear differential equations and vector-valued modular forms. Ramanujan J. 41, 233--267 (2016). 	
  %\bibitem{grabner} P. J. Grabner. Quasimodular forms as solutions of modular differential equations.  Int. J. Number Theory 16 (2020), no. 10, 2233--2274.
%		\bibitem{hurwitz} A. Hurwitz, Adolf: Ueber die Differentialgleichungen dritter Ordnung, welchen
		\bibitem{ka-ko} M. Kaneko; M. Koike, On modular forms arising from a differential equation of hypergeometric type. Ramanujan J. 7(2003), no. 1--3, 145--164.
	%	\bibitem{ka-et-al} M. Kaneko; K. Nagatomo; Y. Sakai. The third order modular linear differential equations. J. Algebra 485 (2017), 332--352.
		\bibitem{ka-za} M. Kaneko; D. Zagier. Supersingular j-invariants, hypergeometric series, and Atkin's orthogonal polynomials. Computational perspectives on number theory (Chicago, IL, 1995), 97--126, AMS/IP Stud. Adv. Math., 7, Amer. Math. Soc., Providence, RI, 1998.
		%\bibitem{klein} F. Klein,  Ueber Multiplicatorgleichungen. (German) Math. Ann. 15 (1879), no. 1, 86--88.
		%\bibitem{milas} A. Milas. Ramanujan’s “Lost Notebook” and the Virasoro algebra. Comm. Math. Phys. 251(2004), no. 3, 657--678.
        \bibitem{mason} G. Mason. Vector-valued modular forms and linear differential operators. Int. J. Number Theory 3 (2007), no.3, 377--390
        \bibitem{mason2} G. Mason. $2$-dimensional vector-valued modular forms. Ramanujan J. 17 (2008), no. 3, 405–427.
		%\bibitem{mukhi} S. Mathur, S. Mukhi, and A. Sen, On the classification of rational conformal field theories. Phys. Lett. B 213(1988), no. 3, 303--308.
		\bibitem{mathann} J. McKay; A. Sebbar.  Fuchsian groups, automorphic functions
		and Schwarzians. Math. Ann. 318 (2), (2000) 255--275.
		%\bibitem{mason} G. Mason, 2-dimensional vector-valued modular forms,
		%Ramanujan J (2008) 17: 405--427.
		%\bibitem{nakaya} T. Nakaya. On modular solutions of certain modular differential equation and supersingular polynomials. Ramanujan J. 48 (2019), no. 1, 13–-20. 
	%	\bibitem{nehari} Z. Nehari, (1949), The Schwarzian derivative and schlicht functions, Bulletin of the American Mathematical Society, 55 (1949) 545--551.
		%	\bibitem{newman} M. Newman. Classification of normal subgroups of the modular group.
		%	\bibitem{ov-ta} V. Ovsienko; S. Tabachnikov. What is the Schwarzian derivative? AMS Notices, 56 (01): 34--36
		%\bibitem{ram} S. Ramanujan, On certain arithmetical functions. Trans. Cambridge Philos. Soc. 22(1916), 159–184.
		%\bibitem{rankin} R. Rankin. Modular Forms and Functions, Cambridge Univ. Press, Cambridge, 1977.
		%\bibitem{s-g} G. Sansone, J. Gerretsen. Lectures on the theory of functions of a complex variable. II,	Geometric theory. Wolters--Noordhoff Publishing, Groningen 1969.
		\bibitem{forum} H. Saber; A. Sebbar. Automorphic Schwarzian equations.  Forum Math. 32 (2020), no. 6, 1621--1636. 
		\bibitem{ramanujan} H. Saber; A. Sebbar. Automorphic Schwarzian equations and integrals of weight 2 forms. Ramanujan J.  57 (2022), no. 2, 551--568. 
		\bibitem{jmaa} H. Saber; A. Sebbar. Equivariant solutions to modular Schwarzian equations, J. Math. Anal. Appl., 508 (2022), no. 2,  Paper No. 125887. 
  \bibitem{preprint} H. Saber; A. Sebbar. Modular differential equations and algebraic systems. arXiv:2302.13459
		%\bibitem{critical} A. Sebbar; H. Saber. On the critical points of modular forms.  J. Number Theory 132 (2012), no. 8, 1780--1787.
		%\bibitem{vvmf} A. Sebbar; H. Saber. Equivariant functions and vector-valued modular forms. Int. J. Number Theory 10 (2014), no. 4, 949--954.
	%	\bibitem{kyushu} A. Sebbar; H. Saber. On the existence of vector-valued automorphic forms. Kyushu J. Math. 71 (2017), no. 2, 271--285.
      %  \bibitem{baus} H. Saber; A. Sebbar.    On the modularity of solutions to certain differential equations of hypergeometric type. Bull. Aust. Math. Soc. 105 (2022), no. 3, 385--391.
		%\bibitem{shimura} G. Shimura; Introduction to the Arithmetic Theory of Automorphic Functions, Princeton University Press, Princeton, New Jersey, 1971.
 % \bibitem{vdp} B. Van der Pol; On a non-linear partial differential equation satisfied by the logarithm of the Jacobian theta-functions, with arithmetical applications. I, II. Nederl. Akad. Wetensch. Proc. Ser. A. 54 = Indagationes Math. 13, (1951). 261--271, 272--284.
		%	\bibitem{yang} Y. Yang. Schwarzian differential equations and Hecke eigenforms on Shimura curves. 	Compos. Math. 149 (2013), no. 1, 1--31.
	\end{thebibliography}
\end{document}